&pdflatex

\documentclass[12pt]{amsart}
\usepackage{amssymb,latexsym}
\usepackage{enumerate}
\usepackage{graphicx}
\usepackage{url}

\makeatletter
\@namedef{subjclassname@2010}{%
  \textup{2010} Mathematics Subject Classification}
\makeatother

\newtheorem{thm}{Theorem}[section]

\newtheorem{lem}[thm]{Lemma}

\newtheorem{prop}[thm]{Proposition}

\theoremstyle{definition}

\newtheorem{rem}[thm]{Remark}

\numberwithin{equation}{section}

\newcommand{\Kcal}{{\mathcal K}}

\newcommand{\CC}{\mathbb{C}}

\newcommand{\QQ}{\mathbb{Q}}

\newcommand{\ZZ}{\mathbb{Z}}

\newcommand{\mand}{ \mathcal{M}_{d,p}}

\begin{document}

%%%%% To ease editing, for IMPAN journals add:

\baselineskip=17pt

%%%%%%%%%%%

%% In the running head, replace first names by initials 
%% and give an abbreviation of the title.

\title[Bounds on the Radius of the p-adic Mandelbrot Set]
{Bounds on the Radius of the p-adic Mandelbrot Set}
\author[Jacqueline Anderson]{Jacqueline Anderson}
\address{Mathematics Department, Box 1917 \\
         Brown University, Providence, RI 02912 USA}
\email{jackie@math.brown.edu}

\date{}

\begin{abstract}
Let $f(z) = z^d + a_{d-1}z^{d-1} + \dots + a_1z \in \CC_p[z]$ be a degree $d$ polynomial. We say $f$ is \emph{post-critically bounded}, or PCB, if all of its critical points have bounded orbit under iteration of $f$. It is known that if $p\geq d$ and $f$ is PCB, then all critical points of $f$ have $p$-adic absolute value less than or equal to 1. We give a similar result for $\frac12d \leq p <d$. We also explore a one-parameter family of cubic polynomials over $\QQ_2$ to illustrate that the $p$-adic Mandelbrot set can be quite complicated when $p<d$, in contrast with the simple and well-understood $p \geq d$ case.
\end{abstract}

\subjclass[2010]{Primary 11S82; Secondary 37P05}

\keywords{ p-adic Mandelbrot set, non-Archimedean dynamical systems}

\maketitle

\section{Introduction}

In complex dynamics, the Mandelbrot set is a source of inspiration for
much current research. This set, 
\[
\label{mandel}
  \mathcal{M}=\{c \in \CC: \text{the
  critical orbit of $f_c(z)=z^2+c$ is bounded}\},
\]
is a complicated and interesting subset of the moduli space of degree
two polynomials. In the past two decades, much research has been done on dynamical systems in a nonarchimedean setting. See, for example, \cite{arxiv0312034,  arxiv12011605, hsia:periodicpoint, MR2040006}.  For a survey of the subject, see \cite{RBnotes} or \cite{MR2316407}. If one examines the Mandelbrot set over a $p$-adic field, one finds the object to be much less inspiring. For any prime
$p$, the $p$-adic Mandelbrot set for quadratic polynomials as defined above, replacing $\CC$ with $\CC_p$, is simply the unit disk.
But when we consider an analogous set for polynomials of higher degree, the $p$-adic Mandelbrot set for $p<d$ can have a complicated and interesting structure.

Let $\mathcal{P}_{d,p}$ denote the parameter space of monic
polynomials $f$ of degree $d$ defined over $\CC_p$ with $f(0)=0$.
Note that every degree $d$ polynomial can be put in this form via
conjugation by an affine linear transformation. We call a map $f$
\emph{post-critically bounded} (PCB) if all of its critical points
have bounded orbit under iteration of $f$. Let $\mand$ denote the
subset of $\mathcal{P}_{d,p}$ that is PCB.  
We define the following quantity, which measures the critical radius of the $p$-adic Mandelbrot set in $\mathcal{P}_{d,p}$:
\begin{equation}
  \label{eqn:defofrdp}
  r(d,p) =\sup_{f\in\mand} \max_{\substack{c\in\CC_p\\ f'(c)=0\\}} \{-v_p(c)\}.
\end{equation}

\begin{rem} One may wonder why we define $r(d,p)$ using the $p$-adic valuations of critical points of polynomials in $\mathcal{P}_{d,p}$ rather than using the $p$-adic valuations of their coefficients. When $p>d$, as we will see in Theorem~\ref{pgtd}, the two notions are equivalent. In other situations, however, parameterizing by the critical points rather than by the coefficients is more natural. For example, when $d$ is a power of $p$, it is easier to describe $\mand$ in terms of the critical points rather than the coefficients, as we will see in  Proposition~\ref{primepower}. This is because there is one uniform bound on the absolute value of critical points for polynomials in $\mand$, but the bounds for the coefficients $a_i$ vary depending on the $p$-adic valuation of $i$.
\end{rem}

Knowing $r(d,p)$ can be useful in searching for all post-critically finite polynomials over a given number field, as is done for cubic polynomials over $\QQ$ in \cite{MR2885981}. For small primes, in particular $p<d$, the set $\mand$ may be complicated and have a fractal-like boundary. We use~$r(d,p)$ as a way to measure its complexity. Just as the critical values for quadratic polynomials in the classical Mandelbrot set over $\CC$ are contained in a disk of radius 2 \cite[Theorem 9.10.1]{beardon:gtm}, the critical points for polynomials in $\mand$ are contained in a disk of radius $p^{r(d,p)}$. For $p>d$ or $d=p^k$, it is known that $r(d,p) = 0$, but for lack of a suitable reference we will provide an elementary proof. The following is the main result of this paper, which gives the exact value of~$r(d,p)$ for certain values of~$p<d$.

\begin{thm}\label{mainthm} 
For $\frac12d< p<d$ we have
\[
  r(d,p) = \frac{p}{d-1}.
\]
Further, for $p=\frac12d$ we have $r(d,p)=0$.
\end{thm}

It may also be interesting to pursue such questions in Berkovich space. For some work related to critical behavior for polynomials in Berkovich space, see \cite{arxiv0909.4528}.

In Section 2, we describe the notation and tools used throughout this paper. Section 3 consists of some lemmas that are frequently employed in the proofs that follow. In Section 4, we discuss the known results in this realm and provide elementary proofs for when $p>d$ or $d=p^k$. We prove our main result in Section 5. Finally, we conclude the paper in Section 6 with a study of a one-parameter family of cubic polynomials over $\CC_2$ to illustrate the fact that $\mand$ can indeed be quite complicated.

\section{Notation and Tools}
Throughout this paper, we fix a prime number~$p$ and we let
\[
  f(z) = z^d+a_{d-1}z^{d-1} + \dots  + a_1z \in \mathcal{P}_{d,p}
\]
be a degree $d$ polynomial in $\CC_p[z]$.  We suppress the $p$ from
our notation for absolute values and valuations. We denote the
critical points of $f$ by $c_1, \dots,  c_{d-1}$, not necessarily
distinct, labeled so that 
\[
  |c_1|\ge|c_2|\ge\dots\ge|c_{d-1}|.
\]
We denote the closed disk centered at $a$ of radius $s$ in
$\CC_p$ by 
\[
  \bar{D}(a,s) = \bigl\{ z\in\CC_p : |z-a|\le s \bigr\}.
\]
The \emph{filled Julia set of $f$} is the set
\[
  \Kcal_f
  = \{ z\in\CC_p : \text{the $f$-orbit of $z$ is bounded} \}.
\]
We let~$R\ge0$ be the smallest number such that
\begin{equation}
  \label{eqn:KsubsetD}
  \Kcal_f \subseteq \bar{D}(0,p^R).
\end{equation}
Equivalently, as shown in \cite{RBnotes}, we can define $R$ as follows:
\[
  R = \max_{1 \leq i \leq d-1} \left\{\frac{-v(a_i)}{d-i}\right\}.
\]
We also set
\begin{equation}
  \label{eqn:reqvc1}
  r = -v(c_1).
\end{equation}
We will often use the fact that 
\[
  a_i = (-1)^{d-i}\frac{d}{i}\sigma_{d-i},
\]
where $\sigma_j$ denotes the $j^{\text{th}}$ symmetric function of the
critical points. 

Whenever we count critical points, roots, or periodic
points for $f$, we do so with multiplicity.

The Newton polygon is a useful object in $p$-adic analysis that we will use frequently. Consider a polynomial
\[
g(z) = \displaystyle\sum_{i=0}^nb_iz^i.
\]
The Newton polygon for $g$ is the lower convex hull of the set of points $\{(i, v(b_i))\}$. If any $b_i = 0$, that point is omitted. (One can think of that point as being at infinity.) This object encodes information about the roots of $g$. In particular, it tells us that $g$ has $x$ roots of absolute value $p^m$ if the Newton polygon for $g$ has a segment of horizontal length $x$ and slope $m$. For proofs of these facts, see~\cite{MR754003}. 

One consequence of these facts is that for polynomials, or more generally, for power series over $\CC_p$, a disk in $\CC_p$ is mapped everywhere $n$-to-$1$ (counting with multiplicity) onto its image, which is also a disk. The following proposition, whose proof can be found in \cite[Corollary 3.11]{RBnotes}, will prove useful.

\begin{prop}\label{diskbehavior}
Let $f(z) = \sum_{i=0}^d b_i(z-a)^i \in \CC_p[z]$ be a degree $d$ polynomial and let $D=\bar{D}(a,p^s)$ be a disk in $\CC_p$. Then $f(D)=\bar{D}(f(a),p^r)$, where
\[
r = \max_{1 \leq i \leq d} \{si-v(b_i)\}.
\]
Moreover, $f: D \to f(D)$ is everywhere $m$-to-$1$ for some positive integer $m$, counting with multiplicity.
\end{prop}

%%%%%%%%%%%%%%%%%%%%%%%%%%%%%%%%%%%%%%%%%%%%%%%%%%%%%%%%%%%%

\section{Preliminary Lemmas}

\begin{lem}\label{Risr} 
Let $f \in \mand$, and let $r$ and $R$ be as defined by
\eqref{eqn:KsubsetD} and \eqref{eqn:reqvc1}.
If $r >0$ and $p>\frac12d$, then $R=r$. 
\end{lem}

\begin{proof}
First note that if $f$ is post-critically bounded, then $R \geq r$ is
necessary. Recall that
\begin{equation}
  \label{eqn:Reqmax}
  R = \max_{1 \leq i \leq d-1}\left\{\frac{-v(a_i)}{d-i}\right\}.
\end{equation}
Since
\[
  \text{$|a_i| = |\sigma_{d-i}|$ for $i \neq p$,
         and $|a_p| = p^{-1}|\sigma_{d-p}| $,}
\]
the only way that $R$ could be strictly greater than $r$ is if
$-v(a_p)/(d-p)$ is maximal in the formula~\eqref{eqn:Reqmax} for~$R$,
with 
\[
  (d-p)r-1 < -v(\sigma_{d-p}) \leq (d-p)r.
\]
In this case, we see that $R = -v(a_p)/(d-p)$ could be as large as $r +
\frac{1}{d-p}$. But if this is true, then $f(c_1)$ is dominated by a single
term, namely $a_pc_1^p$ with 
\[
  -v(f(c_1)) = -v(a_pc_1^p) = pr + (d-p)R > R,
\]
contradicting the fact that $f$ is PCB. Thus $R=r$.
\end{proof}

\begin{lem}\label{pto1}
Let $f\in\CC_p[x]$ be a degree $d$ polynomial, let $\bar{D}(a,s)$ be a disk
in $\CC_p$, and let $m$ be an integer with $p\nmid m$.  If~$f$ maps
$\bar{D}(a,s)$ $m$-to-1 onto its image, then $\bar{D}(a,s)$ contains exactly
$m-1$ critical points of~$f$, counted with multiplicity.
\end{lem}

\begin{proof}
Without loss of generality, replace $f$ with a conjugate so
$\bar{D}(a,s)=\bar{D}(0,1)$. Let
\[
  f =\sum_{i=0}^db_iz^i.
\]
Then, counting with multiplicity,
$f(z)-f(0)$ has $m$ roots in the unit disk, which implies that $m$ is
the largest positive integer such that
\[
  v(b_m) = \min_{1 \leq i\le d} v(b_i).
\]
Now consider the Newton polygon for $f'$. Since $m$ is the largest
integer such that $v(b_m)$ is minimal and $p \nmid m$, we see that
$m$ is also the largest integer such that $v(mb_m)$ is minimal among
all $v(ib_i)$. Therefore, the Newton polygon for $f'$ has exactly
$m-1$ non-positive slopes, which implies that there are $m-1$ critical
points, counted with multiplicity, in $\bar{D}(0,1)$.
\end{proof}

%%%%%%%%%%%%%%%%%%%%%%%%%%%%%%%%%%%%%%%%%%%%%%%%%%%%%%%%%%%%%%%%%%%%%%%%%

\section{Motivation: $p>d$}
\begin{thm}\label{pgtd} Let $p>d$. Then $f \in \mathcal{P}_{d,p}$ is PCB if and only if $|c_i| \leq 1$ for all critical points $c_i$ of $f$.
\end{thm}

This is a known result, but since it does not appear in the literature, we present an elementary proof here.

\begin{proof}
First, suppose $|c_i| \leq 1$ for all $i$. Then $|a_i| = |\frac{d}{i}\sigma_{d-i}| \leq 1$ for all $i$, and so $f(\bar{D}(0,1)) \subseteq \bar{D}(0,1)$. Therefore, $f$ is PCB.

Now let $f$ be PCB and suppose, for contradiction, that $-v(c_1) = r > 0$. Let $m$ be maximal such that $-v(c_m) = r$. (In other words, there are exactly $m$ critical points with absolute value $p^r$.) First, we show that there are exactly $m$ roots $z_1, z_2, \dots,  z_m$ of $f$ such that $-v(z_i) = r$.

Since $p>d$, the Newton polygons for $f$ and $f'$ are the same, up to horizontal translation. Thus, the rightmost segment of the Newton polygon for $f$ has the same slope and horizontal length as the rightmost segment of the Newton polygon for $f'$, and therefore, $f$ has exactly $m$ roots $z_i$ such that $-v(z_i) = r$.

Next, we use Lemma \ref{pto1} to reach a contradiction. Consider $f^{-1}(\bar{D}(0,p^r))$. This is a union of up to $d$ smaller disks $\bar{D}(z_i, p^{s_i})$, where the $z_i$ are the roots of $f$. Note that, since $f$ is PCB, each critical point must lie in one of these disks. By Proposition \ref{diskbehavior}, we know that each $s_i \leq r/d$. 

So, each of the $m$ large critical points $c_1, \dots,  c_m$ must lie in the following set:

\[
V = \displaystyle\bigcup_{i=1}^m\bar{D}(z_i, p^{s_i}).
\] 

$V$ is a disjoint union of $n \leq m$ disks. Relabel the subscripts so that we can express $V$ as follows:
\[
V=\displaystyle\coprod_{i=1}^n \bar{D}(z_i, p^{s_i}).
\]
Let $\bar{D}(z_i, p^{s_i})$ map $d_i$-to-1 onto $\bar{D}(0,r)$. Then, since $V$ contains exactly $m$ preimages of 0, counted with multiplicity, we have
\[
\displaystyle\sum_{i=1}^n d_i = m.
\]
Let $b_i$ be the number of critical points in $\bar{D}(z_i, p^{s_i})$. Then, Lemma~\ref{pto1} tells us that $b_i = d_i-1$, and so the number of critical points, counted with multiplicity, in $V$ is 
\[
\displaystyle\sum_{i =1}^nb_i = m-n < m.
\]
This is a contradiction. Thus, if $f$ is PCB, all the critical points lie in the unit disk.
\end{proof}

In particular, Theorem~\ref{pgtd} implies that $r(d,p)=0$ for $p > d$. The same is also true if $d=p^k$ for some
positive integer $k$. This result follows immediately upon comparing the Newton polygons for $f$ and $f'$. A proof of this result can be found in \cite{MR2881322}, but we present a proof below that is simple and tailored to the normal form used in this paper.

\begin{prop}\label{primepower} Let $d=p^k$ for some positive integer $k$, and let $f \in \mathcal{P}_{d,p}$. Then $f$ is PCB if and only if $|c_i| \leq 1$ for all critical points $c_i$ of $f$.
\end{prop}

\begin{proof}
First, suppose all the critical points for $f$ lie in the unit disk. Then all the coefficients of $f$ are $p$-integral, and so $f(\bar{D}(0,1)) \subseteq \bar{D}(0,1)$. Therefore, $f$ is PCB.

Now, suppose $f$ is PCB. By comparing the Newton polygons for $f$ and $f'$, one sees that the slope of the rightmost segment of the Newton polygon for $f'$ is greater than the slope of the rightmost segment of the Newton polygon for $f$. In other words, the largest critical point of $f$ is strictly larger than the largest root. If this critical point $c$ were outside the unit disk, then $|f(c)| = |c|^d > R$, and $f$ would not be PCB. Therefore, $f$ is PCB if and only if all critical points lie in the unit disk.
\end{proof}

%%%%%%%%%%%%%%%%%%%%%%%%%%%%%%%%%%%%%%%%%%%%%%%%%%%%%%%%%%
\section{The Mandelbrot Radius for Primes $\frac12d\le p<d$}

So far, we have seen a few situations in which the $p$-adic Mandelbrot set can be very easily described: it is simply a product of unit disks. For primes smaller than $d$, this is often not the case. In Section 6, we investigate a one-parameter family of cubic polynomials over $\QQ_2$ to illustrate that $\mand$ can be quite intricate. While we cannot hope to describe $\mand$ exactly for $p<d$, we can get a sense of the size of $\mand$ by calculating $r(d,p)$. We now give a lower bound for $r(d,p)$ when $p<d$.

\begin{prop}\label{lowerbound} Suppose that $p<d$ and that $d$ is not a power of $p$. Let $k$ be the largest integer such that $p^k <d$ and let $\ell$ be the largest integer such that $p^\ell | d$. Write $d = ap^k+b$, where $1 \leq a < p$ and $1 \leq b <p^k$. Then,
\[
r(d,p) \geq \displaystyle\frac{a(k-\ell)p^k}{d-1}.
\]
\end{prop}

\begin{proof}
Let $\alpha \in \CC_p$ satisfy the following equation:
\[ \alpha^{d-1} = \displaystyle\frac{d^d}{(-ap^k)^{ap^k}b^{b}}.
\]
Then, the lower bound given in this proposition is realized by the following map:
\[f(z)=z^{b}(z-\alpha)^{ap^k}.
\]
This map has either two or three critical points: $\alpha$, $\frac{b}{d}\alpha$, and possibly 0. (Zero is a critical point if $b \neq 1$.) We choose $\alpha$ above so that $f(\frac{b}{d}\alpha) = \alpha, f(\alpha) = 0,$ and $f(0)=0$. Thus, $f$ is post-critically finite, and therefore PCB, with 
\[
-v(\alpha) = \frac{a(k-\ell)p^k}{d-1}.
\]
\end{proof}

Since $\ell$ is necessarily less than or equal to $k$, this gives a positive lower bound for $r(d,p)$ in most situations in which $p<d$. It does not give a positive lower bound when $d=p^kq$, where $q<p$, because in this situation $\ell=k$ and thus $-v(\alpha)=0$. The next proposition is the beginning of an exploration of that situation.

\begin{prop}\label{twop} Let $f \in \mathcal{P}_{d,p}$ and suppose $d=2p$. Then $f$ is PCB if and only if $|c_i| \leq 1$ for all critical points $c_i$ of $f$.
\end{prop}

\begin{proof}
Proposition~\ref{primepower} proves this statement if $p=2$, so we proceed assuming $p \neq 2$. Again, one direction is straightforward. If all the critical points are in the unit disk, then all the coefficients of $f$ are $p$-integral,  and $f$ is PCB.

Now let $f$ be PCB and suppose for contradiction that $f$ has a critical point outside the unit disk, with $-v(c_1) = r>0$. By
comparing the rightmost segments of the Newton polygons for $f$ and
$f'$, since the rightmost vertex for $f'$ is one unit above the rightmost vertex for $f$, we get that in most cases the largest critical point for $f$ is larger than its largest root. Since we can write $f(z)=\prod_{i=1}^{d}(z-z_i)$, where the $z_i$ are the roots of $f$ counted with multiplicity, in this situation $|c_1-z_i|=|c_1|$ for all $i$, and therefore $|f(c_1)|=|c_1|^d$. More generally, $|f^n(c_1)|=|c_1|^{d^n}$, and thus $f$ cannot be PCB. The only situation in which this does not happen is if the
rightmost segment of the Newton polygon for $f$ has horizontal length
equal to $p$, in which case it is possible for the largest root of $f$
to have the same absolute value as the largest critical point. In this situation,
there are exactly $p$ roots $z_i$ with $-v(z_i) = r$, and there are at
least $p$ critical points $c_i$ with $-v(c_i) = r$. Suppose there are
exactly $k$ such critical points, counted with multiplicity, where $p
\leq k \leq 2p-1$. Now we use Lemma~\ref{pto1} to show that this is only possible if they are all contained in
a disk centered at a root $z_1$ that maps $p$-to-1 onto $\bar{D}(0,
p^r)$. Recall that if $f$ is PCB, then each critical point lies in a disk $\bar{D}(z_i, p^{s_i})$ mapping via $f$ onto $\bar{D}(0, p^r)$, where $f(z_i)=0$. Proposition~\ref{diskbehavior} implies that $s_i \leq \frac{r}{2p}$ for all $i$. Since $s_i<r$, it is necessary that each of the $k$ large critical points lie in one of these disks centered at a root $z_i$ with $-v(z_i)=r$. Since there are at least $p$ such critical points and only $p$ such roots, Lemma~\ref{pto1} implies that this is only possible if there is one disk $\bar{D}(z_1, p^s)$ mapping via $f$ onto $\bar{D}(0, p^r)$ containing all $p$ such roots and all $k$ of the largest critical points.

Writing $c_i = c_1 + \epsilon_i$ for $2 \leq i \leq k$,
we calculate $f(c_1)$: 

\begin{multline*}
  f(c_1) = c_1^{2p}\left(1-\frac{2p}{2p-1}\binom{k}{1} +
  \frac{2p}{2p-2}\binom{k}{2} - \dots + 2p\binom{k}{2p-1}\right) \\
  + \epsilon, \text{where $ -v(\epsilon) < 2pr$.}
\end{multline*}

We will reach a contradiction if the coefficient of $c_1^{2p}$ is a $p$-adic unit, because this would imply that $-v(f(c_1)) = 2pr > r = R$, and thus that $f$ is not PCB. Modulo $p$, the coefficient of $c_1^{2p}$ is congruent to
\[
1-\frac{2p}{p}\binom{k}{p} \equiv 1-2 \equiv -1 \pmod{p}.
\]
This is because $\binom{k}{p} \equiv\left\lfloor\frac{k}{p}\right\rfloor\pmod{p}$ and $p \le k < 2p$.
Thus, we reach the desired conclusion, that $f$ is PCB if and only if all the critical
points lie in the unit disk.
\end{proof}

In particular, Proposition~\ref{twop} implies that $r(2p, p) = 0$. A similar but more elaborate argument shows that $r(3p,p) = 0$ as well, but the techniques used do not generalize to arbitrary $r(kp,p)$.

Now we turn our attention to the case where $\frac12d <p<d$ to prove the remainder of Theorem~\ref{mainthm}.

\begin{proof}

Suppose that $\frac12d <p<d$. Note that Proposition \ref{lowerbound} shows that $r(d,p) \geq \frac{p}{d-1}$. It remains to show that $\frac{p}{d-1}$ is also an upper bound for~$r(d,p)$. Suppose there is a polynomial $f \in \mand$ with a critical point $c_1$ such that $-v(c_1) = r > 0$. Let $d = p+k$, where $1 \leq k \leq p-1$. Lemma \ref{Risr} implies that the critical orbits for $f$ are all contained in $\bar{D}(0,p^r)$. Let $m$ denote the number of critical points with absolute value $p^r$ (with multiplicity), i.e., 
\[
m = \max\{i: -v(c_i)=r\}.
\]

We break the proof into two cases. The first case we consider is $m<p$.

We will refer to $\{c_1, c_2, \dots c_m\}$ as the \emph{large critical points}. Each large critical point must lie in one of the disks in the following set, where $f(z_i) = 0$ and $s_i \leq r/d$:
\[
f^{-1}(\bar{D}(0,p^r)) = \bigcup_{i=1}^d \bar{D}(z_i, p^{s_i}).
\]
By Lemma \ref{pto1}, we must have more than $m$ roots $z_i$ such that $-v(z_i) = r$. Since the Newton polygons for $f$ and $f'$ can only differ at one place (namely, at the $p$th place), this is only possible if there are exactly $k$ roots of $f$ (and at most $k-1$ critical points) with absolute value $p^r$. This implies that $-v(a_p) = kr$. Let $c_{m+1}$ be the largest critical point such that $-v(c_{m+1}) <r$ and let $t=-v(c_{m+1})$. Since $a_p = \frac{d}{p} \sigma_k$, we must have $-v(\sigma_k) =kr-1$, which implies that $t \geq r-1$. Looking at $f(c_{m+1})$, the sole largest term is $a_pc_{m+1}^p$, which implies that 
\[
  -v(f(c_{m+1})) = kr+pt \geq dr-p.
\]
If $f$ is 
 PCB, then $-v(f(c_{m+1})) \leq r$, which gives the inequality $dr-p \leq r$, and the desired bound follows.

Now suppose the number of large critical points is $m \geq p$. Then, by analysis of the Newton polygons for $f$ and $f'$, either $f$ has a root $z_1$ with $-v(z_1)>r$, or $f$ has exactly $m$ roots of absolute value $p^r$. The first possibility does not occur, because if $-v(z_1) > r$, then $z_1$ must be in the basin of infinity, by Lemma \ref{Risr}. This is a contradiction, since $z_1$ is preperiodic, as 0 is a fixed point for $f$. So, the largest root $z_1$ of $f$ satisfies $-v(z_1) = r$ and the number of large critical points is equal to the number of roots of absolute value $p^r$. By Lemma \ref{pto1}, the only way for $f$ to be PCB is if there is a disk $\bar{D}(c_1, p^s)$ mapping $p$-to-1 onto $\bar{D}(0, p^r)$ containing at least $p$ of the large critical points, where $s \leq r/d$ by Proposition \ref{diskbehavior}. We will again divide into two cases.

First, suppose $-v(c_i-c_j) \leq \max\{0, s\}$ for all critical
points $c_i, c_j$. Let $c_i = c_1 + \epsilon_i$, where $-v(\epsilon_i)
\leq \text{max} \{0, s\}$. Then we have
\[
  f(c_1) = c_1^d-\frac{d}{d-1}\sigma_1c_1^{d-1} + \dots +
  (-1)^{d-1}\frac{d}{1}\sigma_{d-1}c_1.
\]
We will use the fact that
\[
\sigma_i = \binom{p+k-1}{i}c_1^i + \delta_i, \text{where $-v(\delta_i) < ir$}
\]
to simplify our expression for $f(c_1)$ to the following:

\begin{multline}\label{fc1}
  f(c_1) = c_1^d\left(1-\frac{d}{d-1}\binom{d-1}{1} +
  \frac{d}{d-2}\binom{d-1}{2} \right. \\ 
   - \dots +
  \left. (-1)^{d-1}\frac{d}{1}\binom{d-1}{d-1}\right) + \epsilon
\end{multline}

It remains to check that the coefficient of $c_1^d$ is a $p$-adic unit and to determine the largest possible absolute value for $\epsilon$.
First, we look at the coefficient of $c_1^d$ in (\ref{fc1}). This coefficient can be rewritten as follows:

\[
 1-\frac{d}{d-1}\binom{d-1}{1} + \dots + (-1)^{d-1}\frac{d}{1}\binom{d-1}{d-1} = \sum_{i=0}^{d-1} (-1)^i \binom{d}{i}.
\]

Since the full alternating sum from 0 to $d$ of binomial coefficients is always zero, we see that the coefficient of $c_1^d$ is either $1$ or $-1$, depending on whether $d$ is even or odd. Either way, it is a $p$-adic unit, and so the first term in $f(c_1)$ has absolute value $p^{dr}$. Since we must have $-v(f(c_1)) \leq r$ in order for $f$ to be PCB, it is necessary that $-v(\epsilon) = dr$ as well. The only term that can possibly be that large is the one corresponding to $a_pc_1^p$. Let $\sigma_j(\epsilon_i)$ denote the $j^{th}$ symmetric function on the $\epsilon_i$. Then, the portion of $a_pc_1^p$ contributing to $\epsilon$ is:

\begin{multline*}
(-1)^k\frac{d}{p}\left(\binom{d-2}{k-1} \sigma_1(\epsilon_i)c_1^{k-1} + \binom{d-3}{k-2}\sigma_2(\epsilon_i)c_1^{k-2} \right.
\\ + \dots  + \left. \binom{p}{1}\sigma_{k-1}(\epsilon_i)c_1 + \sigma_k(\epsilon_i)\right)c_1^p.
\end{multline*}

Note that since $\binom{d-i-1}{k-i}$ is a multiple of $p$ for all $i<k$, the last term is the only one that can possibly realize the absolute value $p^{dr}$. Looking at $x = (-1)^k\frac{d}{p}\sigma_k(\epsilon_i)c_1^p$, we see that 
\[
  -v(x) \leq pr + 1 + ks \leq 1 + r(p+k/d).
\]
Since $-v(x) = dr$, we have $dr \leq 1 + pr + kr/d$, which implies that $r \leq \frac{d}{k(d-1)}$. This is strictly smaller than $\frac{p}{d-1}$ for $k>1$, and we obtain the desired result. 

We will now treat the $k=1$ case separately. Suppose $d=p+1$ and all $p$ critical points are in a disk centered at $c_1$ of radius $p^s$. The above argument shows that, if $f$ is PCB, then we must have $-v(\sigma_1(\epsilon_i)) = r-1$. We will improve our upper bound for $s$ to prove the result in this case. We know that $\bar{D}(c_1, p^s)$ maps $p$-to-1 onto $\bar{D}(0, p^r)$. Writing $f(z)$ so that it is centered at $c_1$, we have

\[
f(z) = f'(c_1)(z-c_1) + \frac{f''(c_1)}{2!}(z-c_1)^2 + \dots  + \frac{f^{(p)}(c_1)}{p!}(z-c_1)^p + (z-c_1)^{p+1}.
\]

This allows us to see that, by Proposition \ref{diskbehavior},
\[
r = \max_{1 \leq i \leq d} \{-v(f^{(i)}(c_1)/i!) + is\}.
\]

In particular, $r \geq ps + 1 -v(f^{(p)}(c_1))$, which implies that $s$ satisfies the inequality
\begin{equation}\label{sbound}
s \leq \frac{r-1+v(f^{(p)}(c_1))}{p}.
\end{equation}
To calculate $f^{(p)}(c_1)$, we note that
\[
f^{(p)}(z) = (p+1)!z + p!a_p = (p+1)!z - \frac{(p+1)!}{p}\sigma_1.
\]
Plugging in $z=c_1$, we calculate
\[
f^{(p)}(c_1) = (p+1)!c_1 - \frac{(p+1)!}{p}(pc_1 + \sigma_1(\epsilon_i)) =  -\frac{(p+1)!}{p}\sigma_1(\epsilon_i).
\]
This implies that $-v(f^{(p)}(c_1)) = -v(\sigma_1(\epsilon_i)) = r-1$. Plugging this into (\ref{sbound}), we get that $s \leq 0$. Finally, since $r-1 = -v(\sigma_1(\epsilon_i)) \leq s \leq 0$, we see that $r \leq 1$, as desired.

Now we return to the general situation to treat the final case, in which there are at least $p$ large critical points and there exist $c_i, c_j$ such that $-v(c_i-c_j) \geq \max \{0,s\}$. Without loss of generality, let $c_i = c_1$, where $c_1$ is in the disk $\bar{D}(c_1, p^s)$ which contains at least $p$ critical points. There must be a fixed point $\alpha$ for $f$ such that $-v(c_1-\alpha) \leq 0$. We see this by examining the equation
\[ 
f(c_1)-c_1 = \prod_{i=1}^d (c_1-\alpha_i).
\]
Here the $\alpha_i$ are the fixed points for $f$. Since the left hand side of this equation must have absolute value at most $p^r$, the same must be true of the right hand side. Since 0 is a fixed point, we can let $\alpha_d=0$. Then, since $-v(c_1-0) = r$, we are left with
\[
  -v\left(\prod_{i=1}^{d-1}(c_1-\alpha_i)\right) \leq 0.
\]
This implies that there is some $\alpha_i$ satisfying $-v(c_1-\alpha_i) \leq 0$. Call this fixed point $\alpha$.

Next, conjugate $f$ by the affine linear transformation $\phi(z) = z + \alpha$. The new map $f^{\phi} = 
 \phi^{-1}\circ f\circ\phi$ is of the desired form (monic with $f(0)=0$) and is PCB because $f$ is PCB. Note that $f^\phi$ has at least $p$, but no more than $d-2$, critical points in $\bar{D}(0,p^t)$, where $t=\max\{s,0\}$. This implies that the number of large critical points for $f^\phi$ is strictly less than $p$. We have already dealt with this case, and so we know that all the critical points $\gamma_i$ for $f^\phi$ satisfy  $-v(\gamma_i) \leq \frac{p}{d-1}$. So, we can conclude that 
\[-v(\gamma_i) = -v(c_i - \alpha) \leq \frac{p}{d-1}
\]
for all critical points $c_i$ of $f$. If any $c_i$ satisfies $-v(c_i) \leq \frac{p}{d-1}$, we can conclude that $-v(\alpha) \leq \frac{p}{d-1}$ as well, and we reach the desired conclusion. However, if all critical points have the same absolute value as $\alpha$, the result does not yet follow. 

Suppose that $-v(c_i) = r$ for all $i$ and that there exists $c_i$ such that $c_i \not\in \bar{D}(c_1, p^s)$. We have just shown that $-v(c_i-\alpha) \leq \frac{p}{d-1}$, which implies that $-v(c_i-c_j) \leq \frac{p}{d-1}$ for all $i, j$. Therefore, we can proceed assuming that $s < \frac{p}{d-1}$. Since not all critical points are in $\bar{D}(c_1, p^s)$,  there is another disk in the set $V=f^{-1}(\bar{D}(0,p^r))$ containing $n \geq 1$ critical points and (by Lemma \ref{pto1}) $n+1$ roots. Thus, since all the large critical points and roots must be contained in $V$ in accordance with Lemma \ref{pto1}, the  number of roots of absolute value $p^r$ outside $\bar{D}(c_1,p^s)$ must exceed the number of critical points outside $\bar{D}(c_1, p^s)$ by at least one, and therefore, the number of critical points inside $\bar{D}(c_1,p^s)$ must be greater than $p$, as there are exactly $p$ roots in $\bar{D}(c_1,p^s)$.

 Let $c_1, \dots, c_{p+1} \in \bar{D}(c_1,p^s)$ and write all critical points as before, with $c_i = c_1+\epsilon_i$. Here, $\epsilon_i \leq s \leq \frac{r}{d}$ for $2 \leq i \leq p+1$, and $\epsilon_j \leq \frac{p}{d-1}$ for $j>p+1$. Looking at equation \eqref{fc1}, we examine the size of $\epsilon$. Once again, in order to have the necessary cancellation with the leading term, it must be true that $-v(\epsilon) = dr$, which can only be achieved if $-v(\sigma_k(\epsilon_i)) = kr-1$. But,
\[
-v(\sigma_k(\epsilon_i)) \leq (k-2)\frac{p}{d-1} + 2s \leq (k-2)\frac{p}{d-1} + \frac{2r}{d}.
\]
 This gives the following inequality:

$$ kr-1 \leq (k-2)\frac{p}{p+k-1}+\frac{2r}{p+k}$$

This reduces (after a bit of algebra) to: 

$$r \leq \frac{(p+k)(k-1-p+pk)}{(p+k-1)(pk+k^2-2)} \leq \frac{p}{p+k-1} = \frac{p}{d-1},$$ which is the desired result.
\end{proof}

%%%%%%%%%%%%%%%%%%%%%%%%%%%%%%%%%%%%%%%%%%%%%%%%%%%%%%%%%%%%%

\section{A One-Parameter Family of Cubic Polynomials over $\CC_2$}

We have alluded to the fact that $\mand$ can be an interesting set when $p<d$. The following example of a one-parameter family of cubic polynomials reveals that the boundary of this Mandelbrot set can be complicated and fractal-like.

Consider the following one-parameter family of cubic polynomials, where the parameter $t \in\CC_2$:
\begin{equation}\label{cubic}
f_t(z) = z^3 - \frac{3}{2}tz^2.
\end{equation}
Of the two critical points, one (zero)  is a fixed point and the other is $t$. Note that $t=1$ corresponds to a post-critically finite map (i.e., all critical points are preperiodic), with the free critical point 1 mapping to the fixed point $-\frac12$. 

\begin{prop} \label{bdrybehavior}
Consider the one-parameter family of cubic polynomials in defined in \eqref{cubic}. For the sequence of parameters $t_k = 1+2^{2k}$ converging to $t=1$, the corresponding polynomials $f_{t_k}$ for $k \geq 2$ are not PCB. There is another sequence, $t_m = 1 + 3\cdot2^{2m+1}$, also converging to $t=1$, for which the corresponding polynomials for $m \geq 2$ are all PCB.
\end{prop}

Proposition~\ref{bdrybehavior} shows that $t=1$ is on the boundary of the $p$-adic Mandelbrot set for this family of polynomials, in that it is arbitrarily close in the parameter space to parameters corresponding to both PCB and non-PCB maps. For $p>d$, such examples do not exist, as $\mathcal{M}_{d,p}$ is simply the unit polydisk in $\mathcal{P}_{d,p} \simeq \mathbb{C}_p^{d-1}$, which has empty boundary .

\begin{figure}
\begin{center}
\includegraphics[width=11cm]{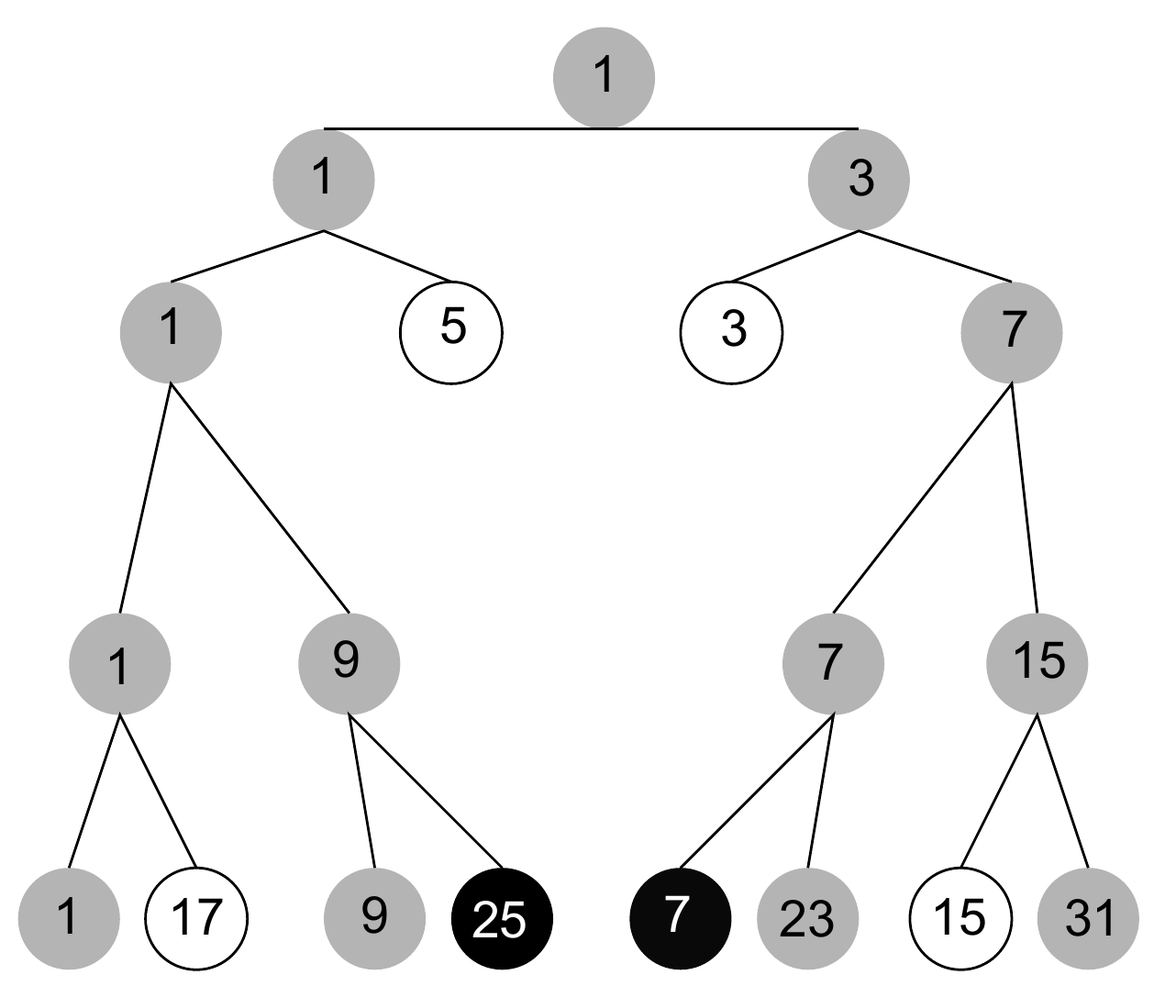}
\caption{Critical orbit behavior for $f_t$ with $|t|=1$.}
\label{fig:fig1}
\end{center}
\end{figure}

\begin{proof}
First, consider the sequence $(t_k)$, where $t_k = 1+2^{2k} \in \CC_2$. As $k$ approaches infinity, $t_k$ approaches $1$ in $\CC_2$. We will now show that for $k \geq 2$, the orbit of the critical point $t_k$ under $f_{t_k}$ is unbounded. In fact, we can say something stronger. If $t \equiv 1 + 2^{2k} \pmod{2^{2k+1}}$ for some $k \in \ZZ$ with $k \geq 2$, then $f_t$ is not PCB. Let $t \in \CC_2$ and $k \geq 2$ such that $t \equiv 1+2^{2k} \pmod{2^{2k+1}}$.

We begin by calculating the first few iterates of $t$ under $f_t$:
\[
f_t(t) = -\frac12 t^3 \equiv -\frac12 \pmod{2^{2k-1}}
\]
\[
f_t^2(t) = -\frac18 t^9 - \frac38 t^7 \equiv -\frac12 \pmod{2^{2k-2}}
\]
\begin{multline*}
f_t^3(t) =-\frac{1}{512}t^{15}(t^{12}+9t^{10}+27t^8+27t^6+12t^4+72t^2 + 108)
\\
\equiv -\frac12 \pmod{2^{2k-4}}.
\end{multline*}
From here, each iterate moves further away from $-\frac12$, so that for $2 \leq i \leq k$ we have
\[
v\left(f_t^i(t)+\frac12\right) = 2k-2i+2.
\]
Thus, $v(f_t^k(t) + \frac12) = 2$ and we can write $f_t^k(t) = -\frac12 + 4u$, where $|u|=1$. A quick calculation then shows that $|f^2(-\frac12 + 4u)| = 4$. Note that in this case, $R = 1$. So, since $f_t^{k+2}(t) \not\in \bar{D}(0,2^R)$, the orbit of $t$ is unbounded.

Now we turn our attention to the other sequence. Let $t_m = 1 + 3\cdot2^{2m+1}$. We will show that the orbit of $t_m$ under iteration of $f_{t_m}$ is bounded. For ease of notation, let $t=t_m$ for some $m \geq 2$. Once again, we begin by calculating the first few iterates:
\[
f_t(t) \equiv -\frac12 - 9\cdot2^{2m} \pmod{2^{4m+1}}
\]
\[
f_t^2(t) \equiv -\frac12 -45\cdot2^{2m-1} \pmod{2^{4m-1}}
\]
\[
f_t^3(t) \equiv -\frac12 -423\cdot2^{2m-3} \pmod{2^{4m-3}}.
\]
In general, for $3 \leq i \leq m+1$, we have
\[
f_t^i(t) \equiv -\frac12 -c_i\cdot 2^{2m-2i+3} \pmod{2^{2m-2i+7}}, \text{ where $c_i \equiv 7 \pmod{8}$.}
\]
More specifically, $c_i = 9(c_{i-1} + 2^{2i-3})$. This shows that we can write $f_t^{m+1}(t) = -\frac12 -2(7+8u)$ for some odd integer $u$. Calculating one more iterate, we see that
\[
f_t^{m+2}(t) = f_t(-\frac{29}{2} - 16u) \equiv 0 \pmod{4}.
\]
This puts $t$ in the basin of attraction of $0$, because $f_t$ maps $\bar{D}(0, \frac12)$ onto itself.

A more careful calculation shows that this proof can be extended to any $t$ such that $t \equiv 1 + 3\cdot2^{2m+1} \pmod{2^{2m+3}}$.
\end{proof}

\begin{figure}
\begin{center}
\includegraphics[width=13cm]{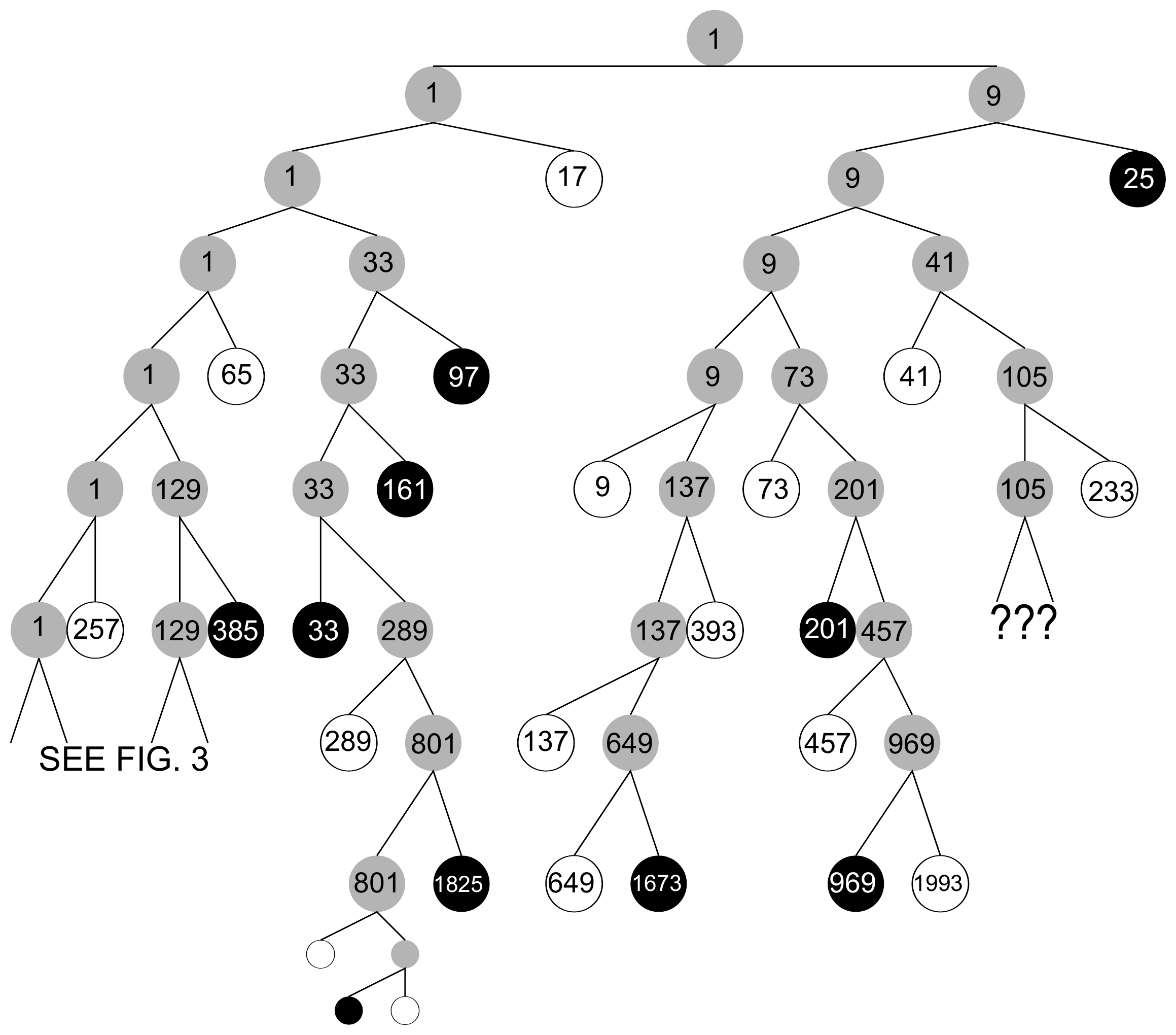}
\caption{Critical orbit behavior for $f_t$ with $t \in \bar{D}(1, \frac18)$.}
\label{fig:fig2}
\end{center}
\end{figure}

\begin{figure}
\begin{center}
\includegraphics[width=12cm]{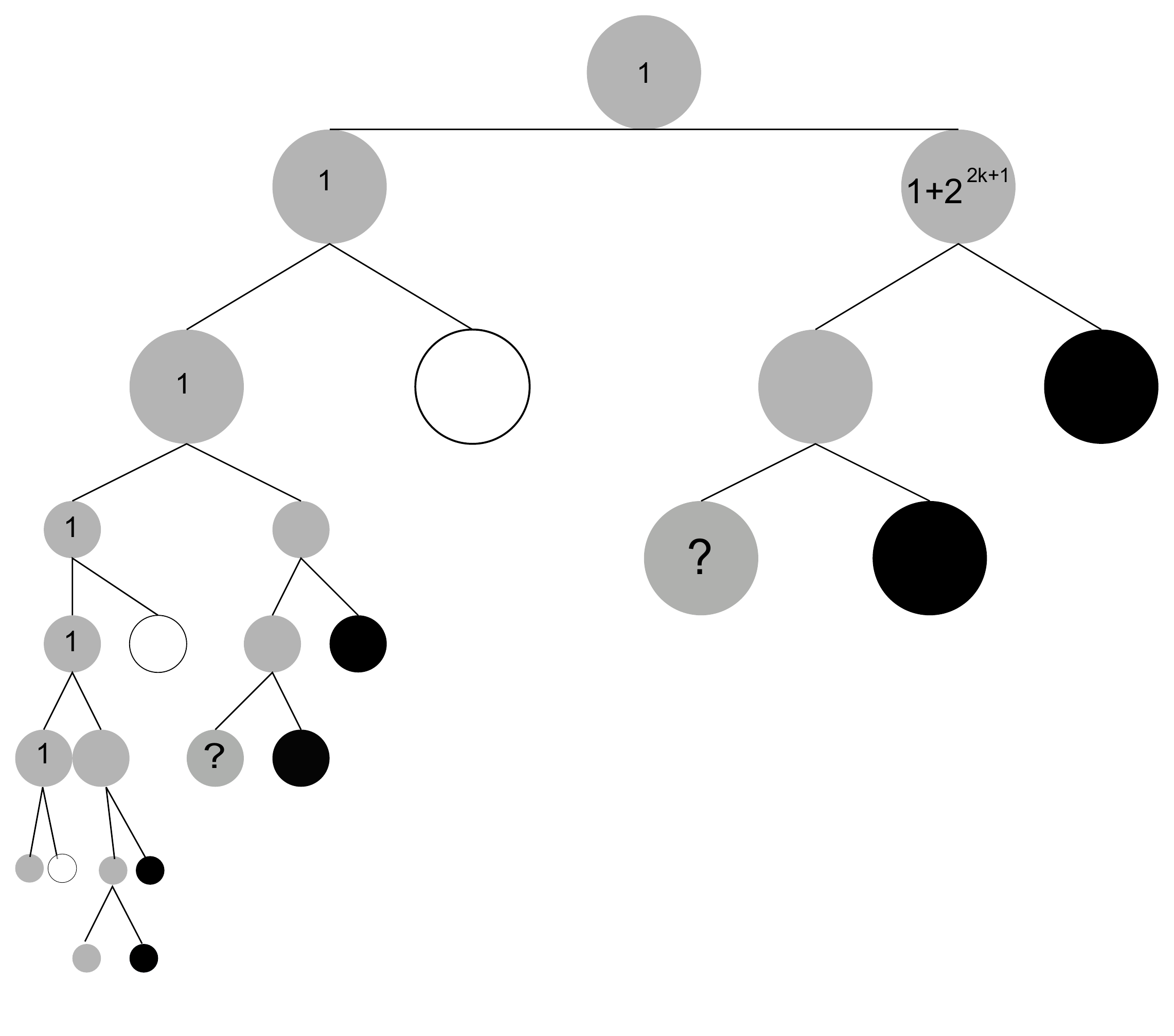}
\caption{Critical orbit behavior for $f_t$ as $t \to 1$. The top node corresponds to the disk $\bar{D}(1, 2^{-(2k+1)})$ for some $k \geq 2$.}
\label{fig:fig3}
\end{center}
\end{figure}

This example shows that this $2$-adic Mandelbrot set, i.e., the set $\{ t \in \CC_2:  f_t \text{ is PCB}\}$, has a complicated boundary. While this set is difficult to visualize over $\CC_2$, we can begin to draw this set if we restrict to $\QQ_2$. Let $\mathcal{M}_t = \{t \in \QQ_2:  f_t \text{ is PCB}\}$. For $|t|>1$, note that $R = -v(\frac32 t) = -v(t)+1$. We now calculate $|f_t(t)|$:
\[
|f_t(t)| = |-\frac12 t^3| =2|t|^3>2|t| = 2^R.
\]
Therefore, the orbit of $t$ is unbounded for all $t$ outside the unit disk. Next note that $f_t$ maps $\bar{D}(0,1)$ into itself for $|t| \leq \frac12$. So $t \in \mathcal{M}_t$ for $|t| \leq \frac12$. If we are interested in the boundary of $\mathcal{M}_t$, we therefore only have to consider $t$ for which $|t|=1$.

We can represent a neighborhood in $\QQ_2$ as a binary tree, as every disk in $\QQ_2$ is comprised of two disjoint disks. For example, the disk $\bar{D}(1,\frac12) = \{t \in \QQ_2: |t|=1\}$, which will be the root of our tree, is comprised of $\bar{D}(1, \frac14)$ and $\bar{D}(3, \frac14)$. Each disk in turn branches into two smaller disks. Traversing down the tree, one ``zooms in'' on a point in $\QQ_2$. See Figure \ref{fig:fig1} for a depiction of the first few levels of this tree. We color a node black if the entire disk is in $\mathcal{M}_t$, we color a node white if the entire disk is outside $\mathcal{M}_t$, and we color a node gray if it contains some points in $\mathcal{M}_t$ and some points outside it. The number that labels each node denotes the center of the disk the node represents. As one moves down the left side of the tree, one zooms in on the post-critically finite boundary point $t=1$. This tree is symmetrical, because $f_t^n(t) = -f_{-t}^n(-t)$, and so $f_t$ is PCB if and only if $f_{-t}$ is PCB. In Figure \ref{fig:fig2}, we depict the disk $\bar{D}(1, \frac{1}{8})$ and the tree that emanates from it to give a sense of the complexity of $\mathcal{M}_t$.

Note that as one zooms in on $t=1$, a self-similar pattern emerges, as illustrated in Figure \ref{fig:fig3} and as shown in Proposition \ref{bdrybehavior}. This is reminiscent of the classical Mandelbrot set over $\CC$ and its fractal-like boundary. Beginning at $\bar{D}(1, 2^{-(2k+1)})$ for any $k>1$, we see in Figure \ref{fig:fig3} that the pattern repeats every time we move two levels down the tree toward $1$. The disk $\bar{D}(1+2^{2k+2}, 2^{-(2k+3)})$ corresponds to non-PCB maps, while the disks $\bar{D}(1+3\cdot2^{2k+1}, 2^{-(2k+3)})$ and $\bar{D}(1+5\cdot2^{2k+1}, 2^{-(2k+4)})$ correspond to PCB maps. The parameter values in the disk $\bar{D}(1+2^{2k+1}, 2^{-(2k+4)})$, labeled in Figure \ref{fig:fig3} with a question mark, exhibit quite complicated behavior. Some data collected with SAGE shows that for many values of $k$ one often has to move twenty or more levels down the tree before one finds a disk that is entirely PCB or entirely non-PCB. Furthermore, the locations of these disks vary with $k$, and there is no apparent pattern. This is where the boundary of $\mathcal{M}_t$ seems to be most intricate.

Our examination of this one parameter family of cubic polynomials in $\QQ_2$ shows that $\mand$ can be quite complicated and interesting for $p < d$ and is an object worthy of further study.

\subsection*{Acknowledgements}
The author would like to thank her advisor, Joseph Silverman, for his guidance and advice, Robert Benedetto for an email exchange that prompted this project, the referee for his or her valuable feedback, and everyone involved in ICERM's semester program in complex and arithmetic dynamics for providing the stimulating research environment in which much of this work was completed.


\begin{thebibliography}{HD}

%% Use the widest label as the parameter.

%% In IMPAN journals, only the title is italicized; boldface is not used.
%% The issue number is only given when the issues are paginated separately.

%%%%%%% To ease editing, use normal size:

\normalsize
\baselineskip=17pt

%%%%%%%%%%%%%%%

\bibitem[1]{beardon:gtm} A.F. Beardon,
\emph{Iteration of Rational Functions},
Graduate Texts in Mathematics, vol. 132, Springer-Verlag, New York, 1991.

\bibitem[2]{arxiv0312034} R.L. Benedetto,
\emph{Wandering domains and nontrivial reduction in non-Archimedean dynamics},
Illinois J. Math. 49 (2005), no. 1, 167-193.

\bibitem[3]{RBnotes} R.L. Benedetto,
\emph{Non-Archimedean dynamics in dimension one: lecture notes},
Available at \url{http://swc.math.arizona.edu/aws/2010/2010BenedettoNotes-09Mar.pdf}, 2010.

\bibitem[4]{arxiv12011605} R.L. Benedetto, P. Ingram, R. Jones, and A. Levy,
\emph{Critical orbits and attracting cycles in p-adic dynamics},
(2012), Preprint. \url{arXiv:1201.1605}.

\bibitem[5]{MR2881322} A. Epstein,
\emph{Integrality and rigidity for postcritically finite polynomials},
Bull. London Math. Soc. 44 (2012), no. 1, 39-46.

\bibitem[6]{hsia:periodicpoint} L.C. Hsia, 
\emph{Closure of periodic points over a non-Archimedean field},
J. London Math. Soc.(2) 62 (2000), no. 3, 685-700.

\bibitem[7]{MR2885981} P. Ingram,
\emph{A finiteness result for post-critically finite polynomials},
Int. Math. Res. Not. (2012), no. 3, 524-543.

\bibitem[8]{MR754003} N. Koblitz,
\emph{p-adic Numbers, p-adic Analysis, and Zeta-Functions},
second ed., Graduate Texts in Mathematics, vol. 58, Springer-Verlag, New York, 1984.

\bibitem[9]{MR2040006} J. Rivera-Letelier,
\emph{Dynamique des fonctions rationnelles sur des corps locaux},
Ast\'erisque (2003), no. 287, xv, 147-230.

\bibitem[10]{MR2316407} J.H. Silverman,
\emph{The arithmetic of dynamical systems},
Graduate Texts in Mathematics, vol. 241, Springer, New York, 2007.

\bibitem[11]{arxiv0909.4528} E. Trucco,
\emph{Wandering Fatou components and algebraic Julia sets},
2010, Preprint. \url{arXiv: 0909.4528}.



\end{thebibliography}
\end{document}